\documentclass[preprint, 11pt]{amsart}
\usepackage{amsfonts,amssymb,amsmath,amsthm,amscd,mathtools,multicol, vwcol,tikz, tikz-cd,caption,enumerate,mathrsfs,geometry,tgpagella,stmaryrd}
\usetikzlibrary{arrows,chains,matrix,positioning,scopes}
\usepackage{inputenc}
\usepackage[foot]{amsaddr}
\usepackage[pagebackref=true, colorlinks, linkcolor=blue, citecolor=red]{hyperref}
\usepackage{fullpage}
\usepackage[all]{xy}

\newcommand{\Tfrac}[2]{%
  \ooalign{%
    $\genfrac{}{}{2pt}1{#1}{#2}$\cr%
    $\color{white}\genfrac{}{}{.4pt}1{\phantom{#1}}{\phantom{#2}}$}%
}
\newtheorem{theorem}{Theorem}[section]
\newtheorem*{theorem*}{Theorem}
\newtheorem{corollary}{Corollary}[theorem]
\newtheorem*{corollary*}{Corollary}
\theoremstyle{remark}
\newtheorem{remark}[theorem]{Remark}
\newtheorem*{remarks*}{Remarks}

\newtheorem{lemma}[theorem]{Lemma}
\theoremstyle{plain}
\newtheorem{proposition}[theorem]{Proposition}

\numberwithin{equation}{subsection}

\makeindex
\begin{document}
	\title{Iterative derivations on central simple algebras}
    \author{Manujith K. Michel and Varadharaj R. Srinivasan}

\address{Indian Institute of Science Education and Research (IISER) Mohali Sector 81, S.A.S. Nagar, Knowledge City, Punjab 140306, India.}
\email{\url{ph20019@iisermohali.ac.in}, \url{ravisri@iisermohali.ac.in}}	
\thanks{The first named author would like to acknowledge the support of CSIR grant 09/0947(13514)/2022-EMR-1}
\thanks{The second named author would like to acknowledge the support of DST-FIST grant SR/FST/MS-I/2019/46(C)}

\begin{abstract}
 We prove that an iterative derivation $\delta_F$ on a field $F$ can be extended to an iterative derivation $\delta_A$ on a central simple $F-$algebra $A$ if the characteristic of $F$ does not divide the exponent of $A$ in the Brauer group of $F.$ For a central simple $F-$algebra with an iterative derivation, we show the existence of a unique (up to isomorphism) Picard-Vessiot splitting field and from the nature its Galois group, we also describe the structure of the central simple algebra in terms of its $\delta_A-$right ideals. 
\end{abstract}

    \maketitle

	\section{Introduction}

An \emph{iterative derivation} on a unital ring $R$ is   a sequence  $\delta_R=(\delta_R^{(n)})_{n\in \mathbb{Z}_{\ge 0} }$ of additive maps on $R$ such that the following conditions hold: 
	\begin{enumerate}[(R1)]
		\item \label{identity-R}$\delta_R^{(0)}$ is the identity map on $R$
		\item\label{product-R} $\delta_R^{(n)}(ab)=\sum_{i+j=n}\delta_R^{(i)}(a)\delta_R^{(j)}(b)$ for all $a,b\in R$ and integer $n\geq 0.$
		
	    \item\label{iterative-R}  $\delta_R^{(n)}\circ \delta_R^{(m)} =\binom{m+n}{n}\delta_R^{(m+n)}.$
        \end{enumerate}
A ring (respectively, field) $R$ with an iterative derivation $\delta_R$ will be called a $\delta-$ring (respectively, $\delta-$field) and sometimes written as $(R, \delta_R).$ The constants of $(R, \delta_R),$ sometimes denoted by $R^\delta,$ is the set $\{a\in A:\delta_A^{(n)}(a)=0  \text{ for all } n\ge 1\}$, which is readily seen to be a subring of $R.$ If $R$ is a field then $R^\delta$ is a subfield of $R.$ 

Let $F$ be a $\delta-$field. A $\delta-F-$module $M$ over a $F$ is an $F-$module $M$ with a sequence $\delta_M:=(\delta_M^{(n)})_{n\in \mathbb Z_{\geq 0}}$ of additive maps on $M$ such that  the following conditions hold:
\begin{enumerate}[(M1)]
	\item $\delta_M^{(0)}$ is the identity map on $M$
	\item  $\delta_M^{(n)}(kv)=\sum_{i+j=n} \delta_K^{(i)}(k)\delta_M^{(j)}(v)$ for all $k\in K$ and $v\in M$
	\item \label{iterativity-module}$\delta_M^{(n)}\circ \delta_M^{(m)}=\binom{n+m}{n}\delta_M^{(n+m)}.$
\end{enumerate}

A $\delta-$ring $A$ is said to be a $\delta-$extension of a $\delta-$ring $B$ if $B$ is a subring of $A$ and the restriction of $\delta^{(n)}_A$ to $B$ is $\delta^{(n)}_B$ for all $n\ge 0.$ A $\delta-F-$algebra over $F$ is an $F-$algebra $A$ which is also a $\delta-$extension of $F.$ Note that a $\delta-F-$algebra is also a $\delta-F-$module.

In this article, for a $\delta-$field $F$ with an algebraically closed field of constants, we use the theory of iterative differential modules developed in \cite{Mat-Put} and develop a theory of $\delta-F-$central simple algebras that is independent of the characteristic of $F.$  We establish a connection between pushforward of torsors and projective representations of differential Galois groups and use it to generalize and reprove \cite[Theorem 1.2]{MMVRS-1} for fields of arbitrary characteristic. If $F$ is a differential field, that is, a field with a derivation map, of arbitrary characteristic, then it is known from a theorem of Hochschild (\cite[Theorem 5]{Hoc55}) that the derivation on $F$ extends to a derivation on $A.$ Furthermore, in a recent article (\cite{Michel-Sahu}), it is shown that derivations on $F$ can be extended to any form over $F$ of an $F^\delta-$algebra with smooth automorphism scheme.  Unfortunately, in general, these results do not hold for iterative derivations (see Remark \ref{nonexample}). However, when the characteristic of $F$ does not divide the exponent of $A$ then we show that one can extend a given iterative derivation on $F$ to an iterative derivation of $A$ (see Theorem \ref{extnofiterderiv}). We  also prove  several results concerning the structure of a $\delta-F-$central simple algebra and provide appropriate generalization of \cite[Theorem 1.1]{MMVRS-1} in the context of iterative derivations. It is worthy to observe that every higher derivation of finite rank\footnote{A \emph{higher derivation of rank} $n$ on unital ring $R$ is a finite sequence of additive maps $\delta^(0)_R, \delta^{(1)_R},\dots, \delta^{(n-1)_R}$ on the ring $R$ such that $\delta^{(0)}_A$ is the identity map and $\delta_R^{(m)}(ab)=\sum_{i+j=m}\delta_R^{(i)}(a)\delta_R^{(j)}(b)$ for all $a,b\in R$ and integer $0\leq m\leq n.$} on a ring $R$ can be extended to any azumaya $R-$algebra (See \cite{Roy-Sridharan}).

\section{Extension of iterative derivations on fields to central simple algebras}
Let $F$ be a $\delta-$field. Unless otherwise explicitly mentioned, $F-$algebras and $F-$modules considered are finite dimensional over $F$. The map $\delta^{(1)}_F$ is a derivation on $F.$ Therefore, if $F$ has characteristic $0,$ then it is easily seen that \begin{equation}\label{iterderiv-deriv} \delta^{(n)}_F=\frac{{\delta^{(1)}}^{n}}{n!},\end{equation} where ${\delta^{(1)}}^{n}$ denotes the map $\delta^{(1)}$ composed $n$ number of times. Thus, every iterative derivation on $F$ is uniquely determined by a derivation on $F.$ Now, from \cite[Theorem 6]{Hoc55}, it follows that iterative derivations on $F$ extends to that of a central simple algebra when $F$ is of characteristic $0.$ 


For an $F-$vector space $V,$ let $[V:F]$ denote the dimension of $V$ over $F.$ The Brauer group of $F$ shall be denoted by $\mathrm{Br}(F).$ For a central simple $F-$algebra $A$, let $[A]\in \mathrm{Br}(F)$ denote the Brauer class of $A,$ the order of $[A]$ in $\mathrm{Br}(F)$ is called the \emph{exponent} of $A$ and it is denoted by $\mathrm{exp}_F(A),$   $\mathrm{deg}(A):=\sqrt{[A:F]}$ and $\mathrm{ind}_F(A):=\mathrm{deg}(D),$ where $D$ is the unique division algebra over $F$ such that $[D]=[A].$ Let $\left(L|F,G(L|F),f\right)$ denote a crossed product algebra over $F,$ where $L$ is a finite Galois extension of $F$ having a Galois group  $G(L|F)$ and $f$ is an element of the second Galois cohomology group $\mathrm H^2(G(L|F),L^*).$

\begin{proposition} \label{extn-crossproduct}
Let $F$ be a field of characteristic $p>0$  and $B=(L|F, G(L|F), f)$ be a crossed product algebra over $F$ such that $p$ does not divide $\mathrm{exp}(B)$. Then every iterative derivation $\delta_F$ on $F$ extends to an iterative derivation $\delta_B$ on $B$ such that $\delta_B$ restricts to an iterative derivation on $L$
    \end{proposition}
    
    \begin{proof} 
    
    Let $F_1$ to be the field of constants of the derivation map $\delta^{(1)}_F$ on $F.$ Then from condition (\ref{iterativity-module}), we obtain that $\delta^{(p)}_F$ maps $F_1$ to itself and we obtain that $\delta^{(p)}_F$ is derivation on $F_1.$ Iteratively,  for each $n\geq 1,$ we define $F_n$ to be the field of constants of the derivation map $\delta^{(p^n)}_F$ on $F_{n-1}.$  Thus, we have the following containment of fields  $$F=F_0\supseteq F_1 \supseteq \cdots .$$  
    Clearly, for each $n\geq 1,$ $x^p\in F_n$ for all $x\in F_{n-1}$ and therefore $F_n$ is a  purely inseparable of exponent one over $F_{n-1}$. Let  $B=(L|F, G(L|F), f).$ Since $p$ does not divide the $\mathrm{exp}(B),$ from \cite[proof of Theorem 5]{Hoc55},  we observe that there is a Galois extension $L_1$ of $F_1$ contained in $L$ such that $L$ is the compositum of $F$ and $L_1$ and the restriction map from $G(L|F)$ to $G(L_1/F_1)$
    is an isomorphism. Also, there is an element $f_1\in H^2(G(L_1|F_1),L_1^\times)$ such that the natural map from $(L_1|F_1,G(L_1|F_1),f_1)\otimes_{F_1}F$ to  $(L|F,G(L|F),f)$ is an isomorphism of $F-$algebras. Similarly,   for each $i\geq 0,$ we obtain a
crossed product algebra $B_i:=(L_i|F_i,G(L_i|F_i),f_i)$ over $F_i$ such that $B_i\otimes_{F_i}F_{i-1}=B_{i-1},$ where $B_0:=B.$ Thus, we have  the following filtration \begin{equation}\label{filtration-data}B=B_0\supseteq B_1\supseteq \cdots\end{equation} of $F_i-$algebras such that   $B_i\otimes_{F_i}F=B$ for every integer $i\geq 0.$ 

We shall now show any such filtration leads to an existence of an iterative derivation on $B$  that extends $\delta_F$. Let $\delta^{(0)}_B$ be the identity map on $B.$ For each positive integer $i,$ using the identification $B_i\otimes_{F_i}F=B,$ we define for each $p^{i-1}\leq j\leq p^i-1,$ \begin{equation}\label{definition-iterderiv}\delta^{(j)}_B\left(\sum^m_{k=1}x_k\otimes_{F_i} \alpha_k\right)=\sum^m_{k=1}x_k\otimes_{F_i}\delta^{(j)}_F(\alpha_k).\end{equation}  Since  the constants of the derivation map $\delta^{(p^i)}_F$ on $F_{i-1}$ is $F_i,$ the above maps are well-defined and $\delta_B:=(\delta^{(j)}_B)_{j\in \mathbb Z_{\geq 0}}$ is easily seen to be an iterative derivation on $B$ (See \cite[pp. 340-341]{MvdP03}).
    \end{proof}

    \begin{theorem}\label{extnofiterderiv}
        Let $F$ be of characteristic $p>0$ and  $A$ be a central simple $F-$algebra such that $p$ does not divide the $\mathrm{exp}_F(A).$ Then, every iterative derivation on $F$ extends to that of $A.$ 
    \end{theorem}
    \begin{proof}
Let $D$ be the division $F-$algebra from the Brauer class of $A.$ We shall first prove that $D$ admits an iterative derivation that extends $\delta_F.$ 
It is known that there is an integer $l$ such that $B:=M_l(D)$ is a crossed product $F-$algebra. Moreover, $p$ does not divide the exponent of $B$. From \cite[Proof of Theorem 5]{Hoc55}, we have for each integer $i\geq 0,$ a crossed product $F_i-$algebra $B_i,$ as defined  in Proposition \ref{extn-crossproduct}, such that  $B_i\otimes_{F_i}F=B$ for every integer $i\geq 0.$

Let $\mathrm{exp}_{F_i}(B_i)=n_im_i,$ where $p$ is prime to $n_i$ and $m_i$ is a power of $p.$ Let $k_i$ and $t_i$ be integers such that $m_ik_i+t_i \mathrm{exp}_F(B)=1.$ For each $i\geq 0,$ let $A_i=B^{\otimes^{m_ik_i}_{F_i}}_i.$ We  have the $F-$algebra isomorphisms  $$A_i\otimes_{F_i}F\cong (B_i\otimes_{F_i} F)^{\otimes^{m_ik_i}_F}\cong B^{\otimes^{m_ik_i}_F}.$$ Thus, in the Brauer group $\mathrm{Br}(F),$  we have \begin{equation}\label{power-brauerequivalent}[A_i\otimes_{F_i}F]=[B]^{m_ik_i}=[B]^{1-t_i\mathrm{exp}(B)}=[B].\end{equation} Since $[A_i]^{n_i}$ is the identity in $Br(F_i)$ and $p$ does not divide $n_i,$  we obtain that $p$ does not divide the exponent of $A_i$ in $\mathrm{Br}(F_i).$

 Let $D_i$ be the division algebra such that $[D_i]=[A_i]$ in $\mathrm{Br}(F_i).$ Let $K_i$ be a finite extension of $F_i$ contained in $F.$ Then $K_i$ is purely inseparable and $[K_i:F_i]$ is a power of $p.$ Therefore $\mathrm{ind}_{F_i}(D_i)$ is prime to $[K_i:F_i]$ and it is known that $$\mathrm{ind}_{K_i}(D_i\otimes_{F_i}K_i)=\mathrm{ind}_{F_i}(D_i)$$ and that  $D_i\otimes_{F_i}K_i$ is a division $K_i-$algebra (\cite[Corollary 4.5.11]{GSz}). Allowing $K_i$ to vary over all finite extensions of $F_i$ contained in $F,$ we obtain that $D_i\otimes_{F_i}F$ is also a division $F-$algebra. Since $D_i\otimes_{F_i} F$ belongs to the Brauer class of $A_i\otimes_{F_i}F,$ from (\ref{power-brauerequivalent}), it also belongs to the Brauer class of $B.$ Therefore, $D_i\otimes_{F_i}F=D$ for all $i\geq 0.$ Since $B\supseteq B_1\supseteq B_2\supseteq\cdots,$ we obtain a filtration  $D\supseteq D_1\supseteq D_2\cdots $ as in (\ref{filtration-data}). Thus, $D$ admits an iterative derivation that extends the derivation on $F.$

Let $A$ be any central simple $F-$algebra with $p$ not dividing $\mathrm{exp}_F(A).$ Since $A=M_n(D)$ for the division $F-$algebra $D$ with $[D]=[A]$ in $\mathrm{Br}(F)$ and since $D$ admits a filtration $$D\supseteq D_1\supseteq D_2$$ with $D_i\otimes_{F_i}F=D,$ we obtain $$M_n(D)\supseteq M_n(D_1)\supseteq M_n(D_2)\supseteq \cdots $$ and that $M_n(D_i)\otimes_{F_i}F\cong M_n(D_i\otimes_{F_i}F)\cong M_n(D)$ for each $i\geq 0.$ Now from (\ref{definition-iterderiv}), we obtain an iterative derivation on $A$ that extends $\delta_F.$
    \end{proof}

\begin{remark}\label{nonexample}
Suppose that $F=F_p(t)$ be the rational function field over the finite field with $p$ elements. For each $i\geq 0,$ let $F_i=F^{p^i}$ and observe that $$F_0=F\supset F_1\supset F_2\supset \cdots$$ and that  $[F^{p^i}:F^{p^{i-1}}]=p.$ Then, there is an iterative derivation $\delta_F$ on $F$ \cite[Proposition 2.2]{Mat-Put}. 
 Let $A$ be a division algebra over $F$ of degree $p^n.$ We shall now show that $A$ admits no iterative derivation that extends $\delta_F.$
Suppose on the contrary that $(A,\delta_A)$ is a $\delta-$ring extension of $(F,\delta_F).$ Let $A_0=A$ and for each $i\geq 1,$  $$A_i:=\{x\in A_{i-1}\ | \ \delta^{p^{i-1}}_A(x)=0\}.$$ Then, each $A_i$ is a division $F_i-$algebra and the natural multiplication map provide an isomorphism  $A_i\otimes_{F_i}F\cong A$ of $F_i-$algebras. Any  division algebra $A_i$ over $F_i$ is known to be cyclic (See \cite[Theorem 3]{Albert-1937}): $A=(L|F, G(L|F), f),$ where $G(L|F)$ is a cyclic group and $L=F(y),$ where $y^{p^i}=f\in F_i.$ By the definition of $F_i,$ there is $x\in F$ such that $x^{p^n}=f.$ Then $y\otimes_{F_i}1-1\otimes_{F_i}x$ is nilpotent in $A_i\otimes_{F_i}F.$ Now since $A_i\otimes_{F_i} F\cong A,$ we obtain that $A$ is not a division algebra, a contradiction.
\end{remark}

\section{ Splittings of $\delta-$central simple algebras}

Throughout this section $F$ is a $\delta-$field with an algebraically closed field of constants $F^\delta.$ If $(A,\delta_A)$ and $(B,\delta_B)$ are $\delta-F-$ algebras (respectively, $\delta-F-$ modules) then the tensor product $A\otimes B$ has the  iterative derivation $(\delta_A\otimes \delta_B)^{(n)}(a\otimes b)=\sum_{i+j=n} \delta_A^{(i)}(a)\otimes\delta_B^{(j)}(b),$ making $(A\otimes B, \delta_{A\otimes B})$ a $\delta-F-$ algebra (respectively, a $\delta-F-$module). 
A morphism  $\phi: A\to B$ of $\delta-F-$algebras (respectively, $\delta-F-$modules) is an $F-$morphism such that $\phi(\delta_A^{(n)}(a))=\delta_B^{(n)}(\phi(a)) $ for all $a\in A$ and $n\ge 0.$ Let $(K,\delta_K)$ be an extension field of $(F, \delta_F).$  The $\delta-$field $K$ \emph{splits} a $\delta-F-$algebra (respectively, a $\delta-F-$module) $A$, or equivalently, $(A\otimes_F K, \delta_A\otimes \delta_K)$ is a \emph{split} $\delta-K-$algebra (respectively, $\delta-K-$module) if the natural map $\mu_K:(A\otimes_F K)^\delta\otimes_{F^\delta}K\to A\otimes_F K$ is a $\delta-K-$isomorphism. If $A$ is a $\delta-F-$central simple algebra split by a $\delta-$field $K$ then, since $A\otimes_F K$ is a central simple $K-$algebra, $(A\otimes_FK)^{\delta}$ is a central simple algebra over the algebraically closed field $F^\delta$. Therefore, the $\delta-F-$algebra  $(A\otimes_F K)^\delta\otimes_{F^\delta}K$ is isomorphic to the $\delta-F-$algebra $\mathrm{M}_n(F^\delta)\otimes_{F^{\delta}}(K),$ which in turn is isomorphic to the $\delta-F-$algebra $(\mathrm M_n(K), \Delta_{K}),$ where for any matrix $(a_{ij})\in M_n(K)$ and $l\geq 0,$ $$\Delta^{(l)}_K((a_{ij}))=(\delta^{(l)}_K(a_{ij})).$$

The $\delta-$field extension $K$ of $F$ is said to be  Picard-Vessiot extension for a $\delta-F-$module $M$ if 
\begin{enumerate}[(i)] 
	\item $K^\delta=F^\delta$
	\item \label{muisomorphism} $(M\otimes K, \delta_M\otimes\delta_K)$ is a split  $\delta-K-$module.
	\item \label{generation}If $e_1,\dots, e_n$ is an  $F-$basis of $M$ then $K$ is generated as a field over $F$ by the coefficients of all $v\in (M\otimes K)^{\delta_M\otimes \delta_K}$ with respect to the $K-$basis $e_1\otimes 1, \dots, e_n\otimes 1.$
\end{enumerate}

We shall now recall several facts from the Picard-Vessiot theory of iterative differential modules (\cite{Mat-Put, MvdP03}). For a given $\delta-F-$module $M$, Picard-Vessiot extensions exist and are unique up to $\delta-F-$isomorphisms.  The map $\delta^{(1)}_F$ is a derivation on $F.$ If $F$ has characteristic $0$  then it is known that $K$ is  a Picard-Vessiot extension of the differential field $F$ with a derivation $D$ if and only if the field $K$ is a Picard-Vessiot extension of the $\delta-$field $F,$ where $\delta_F$ is defined as \begin{equation}\label{iterderiv-deriv}\delta^{(n)}_F=\frac{D^{n}}{n!},\end{equation} where $D^{n}$ denotes the map $D$ composed $n$ number of times.  The crucial point here is that  the differential field $(F, D)$ and the $\delta-$field $F$ have the same field of constants. 

Let $K$ be a Picard-Vessiot extension of $F$ for a $\delta-F-$module $M.$
The group $\mathrm{G}(K|F)$ of all $\delta-F-$automorphisms of $K$ is called the \emph{$\delta-$Galois group} of $K.$ The action of $\mathrm{G}(K|F)$ on $M\otimes_F K$ given by 
  \begin{equation}\label{actiononconstants}
      \sigma \cdot(v\otimes x)=v\otimes \sigma(x), \quad\sigma \in G(K|F), \quad v \in M, \quad x\in K
  \end{equation}
  restricts to an action on $(M\otimes K)^\delta$ and yields a faithful representation $\phi_M: G\to \mathrm{GL}_n(M\otimes K)^\delta$ over $F^\delta.$ The image of $\phi_M$ under this representation is a reduced and therefore it is a smooth closed algebraic subgroup of $\mathrm{GL}((M\otimes K)^\delta).$ The largest $F^\delta-$vector subspace of $K$, denoted by $\mathrm{T}(K|F),$ on which $\mathrm{G}(K|F)$ acts rationally forms a $\delta-F-$algebra called the \emph{Picard-Vessiot algebra} or the {Picard-Vessiot} ring of $M$. This action gives an action of $G$ on  $X=\mathrm{maxspec}(T),$ which is an affine algebraic scheme over $K.$ It is known that $X$ is a $G-$torsor over $K.$

\subsection{Pushforward of Torsors.} Here we state few facts concerning pushforward of torsors.
   We start by recalling the Tannakian formalism of torsors.  Let $G$ be an $F^\delta-$group . 
    The $G-$torsors  over $F$ are in bijective correspondence with  fiber functors $\mathscr F: \mathrm{Rep}(G)\rightarrow \mathrm{Vect}_F$ from the category of finite dimensional representations of $G$ over $F^\delta$ to the category of finite dimensional vector spaces over $F$. The functor is given by $Y\mapsto \mathscr F_Y,$ where  $\mathscr F_Y(V)=(V \otimes F[Y])^G.$ It is well known that $\mathrm{dim}_{F^\delta}(V)=\mathrm{dim}_F(\mathscr F_Y(V)).$ Moreover $\mathscr F_Y$ is neutralized by $F[Y]$ in the following sense. The composition functor $V\mapsto \mathscr F_Y(V)\otimes F[Y]$ from $Rep(G)$ to the category of finitely generated projective modules over $F[Y]$ is equivalent to the functor $V\mapsto \mathbf{F}(V)\otimes F[Y]$ from $\mathrm{Rep}(G)$ to the finitely generated projective modules over $F[Y],$  where $\mathbf F$ is the forgetful functor from $\mathrm{Rep}(G)$ to the category of $F^\delta-$vector spaces. That is, we have the following isomorphism of $F[Y]-$modules 
  $$\mathscr F_Y(V)\otimes_F F[Y]\cong \mathbf F(V)\otimes_{F^\delta}F[Y].$$

Let $\phi:G\rightarrow H$ be a homomorphism of $F^\delta-$groups and $Y$ be a $G-$torsor over $F.$ The pushforward of $Y$ along $\phi$ is the  $H-$torsor  corresponding to the fibre functor $\mathscr F_Y \circ \phi^*:Rep(H)\rightarrow Rep(G)\rightarrow Vect_F$ where $\phi^*$ is the obvious functor induced by $\phi.$ Alternatively, the pushforward can be defined as follows. The functor of points of  $(F[Y]\otimes F^\delta[H])^G$ is given by sheafifying the presheaf  $\Tfrac{Y\times H}{G}$ of $G-$orbits, where $G$ acts on $Y\times H$ diagonally. The scheme thus obtained is called the contracted product of $Y$ and $H$ over $G,$ denoted by $Y\times^G H.$ The action of $H$ on $Y\times^G H$ is given by the action of $H$ on the right factor $H$ and under this action, $Y\times^G H$ is an $H-$torsor from  called the pushforward of $Y$ along $\phi$ (See \cite[Appendix]{MR3846058}).

  Given a homomorphism $\psi:G\rightarrow PGL_n=\mathrm{Aut}(M_n),$ the pushforward of $Y$ along $\psi$ provides a $\mathrm{PGL_n}-$torsor $Y\times^{G} \mathrm{PGL_n}$ over $F$ that corresponds to the central simple algebra $(F[Y]\otimes M_n(F^\delta))^G$ over $F$. Note that this $\mathrm{PGL_n}$ torsor's coordinate ring  is $(F^\delta[\mathrm{PGL_n}]\otimes_{F^\delta
  } F[Y])^G,$ where the group $G$ acts diagonally on $F^\delta[\mathrm{PGL_n}]\otimes_{F^\delta} F[Y]$.

  

Let $G:=G(K|F),$  $\{\{M\}\}$ be the rigid monoidal category of $\delta-F-$modules generated by $M$ and $\mathrm{Rep}(G)$ be the category of all representations of $G$ over $F^\delta.$ Then, $\mathscr S: \{\{M\}\}\to \mathrm{Rep}(G),$ given by $\mathscr S(N)=(N\otimes_F K)^\delta,$ for any $N\in \{\{M\}\},$ establishes an equivalence of categories with a quasi-inverse given by $V\mapsto (V\otimes_{F^\delta} K)^G.$   It is worthy to observe that for $X=\mathrm{maxspec}(T(K|F)),$ we have $F[X]=T(K|F)$ and  using a dimension count, we conclude that $\mathscr F_X(V)=(V\otimes_{F^\delta} K)^G.$ Thus, if we compose the quasi-inverse $V\mapsto (V\otimes_{F^\delta} K)^G$ with the forgetful functor from $\{\{A\}\}$ to $\mathrm{Vect}_F$ we obtain $\mathscr F_X.$

\subsection{Splittings of $\delta-F-$algebras} 

Let $A$ be a $\delta-F-$alegbra and  $K$ be a Picard-Vessiot extension of a $\delta-F-$module $A.$ Then, by definition, we have the natural $\delta-K-$module isomorphism \begin{equation}\label{mu-isomorphism}\mu_K:(A\otimes_F K)^\delta\otimes_{F^\delta}K\to A\otimes_F K.\end{equation} This isomorphism  is readily seen to be a $\delta-K-$algebra isomorphism. Thus, the Picard-Vessiot extension $K$ also splits the  $\delta-F-$algebra $A.$ Since $A\otimes_F K$ is a central simple $K-$algebra, it follows that $(A\otimes_FK)^\delta$ is a central simple $F^\delta-$algebra. Thus, the image of the homomorphism $G\rightarrow GL((A\otimes K)^\delta),$ defined in  \ref{actiononconstants}, is contained in $\mathrm{Aut}((A\otimes K)^\delta)=M_n(F^\delta),$ where $n=\mathrm{dim}_{F^\delta}((A\otimes_FK)^\delta)=\mathrm{deg}(A).$ Thus, we obtain a homomorphism of algebraic groups $\mathrm{G}(K|F)\rightarrow PGL_n(F^\delta),$ which we sometimes call a projective representation of degree $n.$ Recall that the action of $\mathrm{G}(K|F)$ on the Picard Vessiot ring $\mathrm{T}(K|F)$
  induces an action of $G$ on $X=\mathrm{maxspec}(T(K|F))$ which makes $X$ a $G-$torsor over $F.$


  
  \begin{theorem} There is  a bijective correspondence between the set of all isomorphism classes of $\delta-F-$central simple algebras split by $K$ and having degree $n$ over $F$ and the set of all inequivalent projective representations of $\mathrm{G}(K|F)$ having degree $n.$ Under this bijection the central simple $F-$algebra  constructed from a projective representation $\phi$ of $\mathrm{G}(K|F)$ corresponds to the $PGL_n-$torsor obtained by the pushforward of $X=maxspec(T(K/F))$ along $\phi.$
  \end{theorem}
  \begin{proof} Let $A$ be a $\delta-F-$central simple algebra of degree $n$ split by $K.$ Then,  $$(A\otimes K)^\delta=(A\otimes F[X])^\delta\cong M_n(F^\delta).$$ Therefore $\mathrm{G}(K|F)$ acts both rationally and also as a group of $F-$automorphisms of $M_n(F^\delta),$ providing a projective representation of degree $n.$ Conversely, let $\phi:G\to \mathrm{PGL_n}(F^\delta)=\mathrm{Aut}(M_n(F^\delta))$ be a map of algebraic groups.  Then $(M_n(F^\delta)\otimes_{F^\delta} F[X])^{G(K|F)}$ is an $F-$algebra. Since $$(M_n(F^\delta)\otimes_{F^\delta} F[X])^{G(K|F)}\otimes_F F[X]\cong M_n(F^\delta)\otimes_{F^\delta} F[X]\cong M_n(F[X]),$$ we obtain $(M_n(F^\delta)\otimes_{F^\delta} F[X])^{G(K|F)}$ is a central simple $F-$algebra. As noted earlier, the $\mathrm{PGL_n}$ torsor $X\times ^{G(K|F)} \mathrm{PGL_n}$ corresponds to the central simple $F-$algebra $(M_n(F^{\delta})\otimes_{\delta} F[X])^{G(K|F)}.$ \end{proof}

  The push forward of an $H-$torsor $X$ along a homomorphism $H\rightarrow G$ can be realised as the image of $X$ under the map $H^1(F,H)\rightarrow H^1(F,G)$ after identifying the isomorphism classes of torsors over $F$ with the first Galois cohomology. Using this fact, we make the following remarks. 
  \begin{remark}
 \item  \begin{enumerate}
     \item \label{trivialtorsor-matrixalgebra} Let the $A$ be a $\delta-F-$central simple algebra split by a Picard Vessiot extension $K.$ If the $\mathrm{G}(K|F)-$torsor $\mathrm{maxspec}(T(K/F))$ is trivial then the pushforward of this torsor along the projective representation is also trivial and therefore the central simple algebra $A,$ corresponding to this trivial $\mathrm{G}(K|F)-$torsor must be a matrix $F-$algebra.\\

  \item The converse of (\ref{trivialtorsor-matrixalgebra}) need not hold in general. However, if the associated projective representation of $\mathrm{G}(K|F)$ embeds as a parabolic subgroup of $\mathrm{PGL_n}(F^\delta)$ then the converse does hold. To see this, let $A$ be a matrix $F-$algebra. By definition, $A$ must correspond to the pushforward of the torsor $\mathrm{maxspec}(T(K|F))$ along the projective representation and therefore, this pushforward must be trivial. But,   in this case,  the map  $H^1(F,G)\rightarrow H^1(F,PGL_n)$ induced by the projective representation is injective \cite[$\S$ 2, Exercise 1]{serre1994cohomologie}. Therefore, $\mathrm{maxspec}(T(K|F))$ must be a trivial torsor for $\mathrm{G}(K|F)$ over $F.$\\
   \item Note that torsors of a general linear group
   are trivial. Therefore, if a  projective representation $\phi$ of a differential Galois group lifts to a linear representation $\tilde{\phi},$ that is, we have the commutative diagram of algebraic groups

   \begin{equation*}\label{algebraicgrouprep}
		\begin{tikzcd}
			& \mathrm{GL}_n(F^{\delta})\arrow[d, rightarrow, "\pi"] \\
			G(K|F)\arrow[r,rightarrow, "\phi"] \arrow[ur, rightarrow, dashed, "\tilde{\phi}"]&\mathrm{PGL}_n(F^\delta)
		\end{tikzcd}
	\end{equation*}
    
then, the  pushforward of the torsor $\mathrm{maxspec}(T(K|F))$ along $\phi$ is trivial and we obtain that $A$ is a matrix $F-$algebra. \\
   
   \item In \cite{article}, authors have a field $F$ of characteristic zero and a  Picard Vessiot extension $K$ such that differential Galois group $\mathrm{G}(K|F)=\mathrm{PGL}_2$ and that $\mathrm{maxspec}(T(K/F))$ is a nontrivial $PGL_2$ torsor. If we take identity map as the  projective representation of $\mathrm{PGL}_2$ then the $\delta-F-$central simple algebra, which is split by $K$ must be a quaternion division $F-$algebra. \\

   \item If $A$ is a central simple $F-$algebra and the characteristic of $F$ does not divide $\mathrm{exp}(A)$ then there is a central simple $F-$algebra $B,$ which is Brauer equivalent to $A,$ admitting an iterative derivation $\delta_B$ such that $(B,\delta_B)$ is split by a finite Galois extension of $F$. This can be seen as follows. Since $F^\delta$ contains all roots of unity there is a crossed product $F-$algebra $B$ that is Brauer equivalent to $A$ such that $B=(K|F,G(K|F),f),$ where  $f$ is a cocycle having values in $F^\delta$ \cite[Lemma 5.3]{MeirEhud}. If $\{u_\sigma\ | \ \sigma\in G(K|F)\}$ is the  $F-$basis of $B$ such that $u_\sigma u_\tau=f(\sigma, \tau) u_{\sigma\tau}$ then the unique iterative derivation on $K$ that extends the iterative derivation on $F$  can be extended further to an iterative derivation on $B$  by setting for all $j\geq 0,$ $$\delta^{(j)}_B\left(k_0+\sum_{\substack{\sigma\in G(K|F) \\ \sigma\neq 1}}k_\sigma u_\sigma\right)=\delta^{(j)}_K(k_0)+\sum_{\substack{\sigma\in G(K|F) \\ \sigma\neq 1}}\delta^{(j)}(k_\sigma) u_\sigma$$
   Note that the condition  (\ref{identity-R}) readily follows and that elements of $\mathrm G(K|F)$ commutes with $\delta_K$. Now, for $\sigma, \tau\in G(K|F)$ and $k_\sigma, k_\tau\in K$ and for nonnegative integers $i,j,n,$  we have \begin{align*}&\delta^{(n)}_B\left(k_\sigma u_\sigma k_\tau u_\tau\right)=\delta^{(n)}_B\left(k_\sigma \sigma (k_\tau) u_{\sigma}u_{\tau}\right)=\delta^{(n)}_K\left(k_\sigma \sigma (k_\tau)f(\sigma,\tau)\right)u_{\sigma\tau}=\delta^{(n)}_K\left(k_\sigma \sigma (k_\tau) \right)f(\sigma,\tau)u_{\sigma\tau}\\ &\delta^{(i)}_B\left(k_\sigma u_\sigma\right) \delta^{(j)}_B\left(k_\tau u_\tau\right)=\delta^{(i)}_K\left(k_\sigma \right)u_\sigma \delta^{(j)}_K\left(k_\tau\right)u_\tau=\delta^{(i)}_K(k_\sigma)\delta^{(j)}_K(\sigma(k_\tau))f(\sigma,\tau)u_{\sigma\tau}.\end{align*} From these equations, we obtain that the maps $\delta^{(j)}_B$ satisfy conditions (\ref{product-R}) and (\ref{iterative-R}).  Let $\{v_1,\dots, v_n\}$ be a $F-$basis of $K.$ Then, for each $1\leq i\leq n,$ the elements $$\sum^{n}_{j=1}v_j\otimes_F v^{(j-1)}_i$$  are readily seen to be constants. Since for each $\sigma\in \mathrm G(K|F),$ $u_\sigma$ is also a constant, we obtain that the natural mapping $$(B\otimes_FK)^\delta\otimes_{F^\delta}K\mapsto B\otimes_F K$$ is a $\delta-F-$isomorphism. Thus, $K$ splits the $\delta-F-$algebra $B.$\\
    
   \item Similar to the profinite absolute Galois group of a field, which is the inverse limit over finite Galois groups, in the context of $\delta-$fields, we have a pro-algebraic absolute $\delta-$Galois group $\mathrm G_F,$ which is the inverse limit over all  $\delta-$Galois groups. The isomorphism classes of  all $\delta-F-$central simple algebras are given by the projective representations of $\mathrm G_F.$ One can define a $\delta-F-$Brauer monoid, similar to differential Brauer monoid in characteristic zero (See \cite{Mag23}), as the direct limit over the direct system of projective representations of $\mathrm G_F.$ Here the ordering in the direct system  is given by divisibility of positive integers and the arrows are given by mapping a projective representation $\phi$ of degree n to the composite $\alpha_{n,m}\circ \phi,$ where $\alpha_{n,m}:\mathrm{PGL_n}\rightarrow \mathrm{PGL_{mn}}$ is the natural map induced by $X\mapsto \otimes^n X$ on matrices.\\

   \item Let $V$ be a $F^\delta-$vector space with a structure satisfying the following conditions i) $\mathrm{Aut}(V),$ the subgroup of $\mathrm{GL}(V)$ of all structure preserving automorphisms is Zariski closed. ii) Each $\mathrm{Aut}(V)-$torsor over $F$ corresponds to a form of $V$ over $F$. Then, for any linear algebraic group $G$ and a morphism $\phi: G\to \mathrm{Aut}(V)$ and for any  $G-$torsor $X$ over $F,$ the pushforward of $X$ along $\phi$ gives an $\mathrm{Aut}(V)-$torsor that corresponds to a form of $V,$ namely $A:=\left(V\otimes_{F^\delta}F[X]\right)^G$ having a structure induced by the one on $V.$ For example, a $F^\delta-$vector space $V$ with a $(p,q)-$tensor is readily seen to satisfy the conditions i) and ii). Now, if $G$ is a $\delta-$Galois group of a Picard-Vessiot extension $K$ of $F$ and $X=\mathrm{maxspec}(T(K|F))$ then $A$ is a $\delta-F-$module, where $\delta_A$ respects the structure on $A$. Furthermore, there is an isomorphism of $\delta-K-$modules $$A\otimes_FK\to V\otimes_{F^\delta} K$$ which also respects the structure on $A\otimes_F K$ and $V\otimes_{F^\delta}K.$ Thus, we obtain a natural extension of the notion of a $\Phi-$object in the sense of \cite[$\S$ 3.3]{Tsui-Man} for field of positive characteristic.  Conversely, If $A$ is a $\delta-F-$module with a compatible structure and $K$ is a Picard-Vessiot extension for $(A, \delta_A)$ then there is a morphism of algebraic group $G(K|F)\to \mathrm{Aut}((A\otimes_FK)^\delta)$, where $\mathrm{Aut}((A\otimes_FK)^\delta)$ denotes the $F^\delta-$automorphisms that preserve the structure on  $(A\otimes_FK)^\delta.$  
\end{enumerate}
   \end{remark}

\section{structure of $\delta-F-$central simple algebras}
   The main objective of this section is to establish a correspondence between the nature of a $\delta-F-$central simple algebra and its $\delta-$Galois group, similar to what was established by the authors in \cite[Theorem 1.1]{MMVRS-1}. In fact, most of the proofs from \cite[$\S 3$ and $\S 4$]{MMVRS-1}, in particular \cite[Propositions 3.1 and 4.1]{MMVRS-1}, carry over to the context of iterative derivations without any need for modifications.  First, we shall recall a few definitions from \cite{MR2178661}  and \cite{MR3293724}.  Let $H$ be a closed subgroup of a connected reductive group $G.$ The group $H$ is called \emph{$G-$completely reducible} if whenever $H$ is contained in a parabolic subgroup $P$ of $G,$ then it is contained in a Levi subgroup of $P.$ The group $H$ is said to be \emph{$G-$irreducible} if $H$ is not contained in a proper parabolic subgroup of $G.$ The group $H$ is said to be \emph{$G-$indecomposable} if there is no parabolic subgroup $P$  of $G$ such that $H$ is contained in a Levi subgroup of $P.$ Note that a subgroup $H$ of $GL(V)$ is $\mathrm{GL(V)-}$completely reducible (respectively, $\mathrm{GL(V)-}$irreducible, $\mathrm{GL(V)-}$indecomposable) iff $V$ is a completely reducible (respectively, irreducible, indecomposable) $H-$module.


   \begin{lemma} \label{PGL_n-cr-irr-ind}Let $H$ be a closed subgroup of $\mathrm{PGL_n}(F^\delta).$ 
    \begin{enumerate}
   \item  \label{H-cr} $H$  is $\mathrm{PGL_n-}$completely reducible if and only if $M_n(F^{\delta})$ is a direct sum of minimal $H-$stable right ideals.
   \item \label{H-irr} $H$ is  $\mathrm{PGL_n-}$irreducible if and only if $\mathrm{M_n}(F^\delta)$ has no proper non trivial $H-$stable right ideals.
   \item \label{H-ind} $H$ is $\mathrm{PGL_n-}$indecomposable  if and only if $\mathrm{M_n}(F^\delta)$ cannot be written as a direct sum of proper $H-$stable right ideals.
   \end{enumerate}
  	
  \end{lemma}
  \begin{proof}
  Let $\pi: \mathrm{GL_n}(F^\delta)\to \mathrm{PGL_n}(F^\delta)=\mathrm{Aut}(\mathrm{M_n}(F^\delta))$ be the canonical projection. Since $\pi$ is non-degenerate, a closed subgroup $H$ of $\mathrm{PGL_n}(F^\delta)$ is $\mathrm{PGL_n}(F^\delta)-$ completely reducible if and only if $\pi^{-1}(H)$ is $\mathrm{GL_n}(F^\delta)-$compeletely reducible. Also, $\pi^{-1}(H)$ is  $\mathrm{GL_n}(F^\delta)-$completely reducible if and only if  $(F^\delta)^n$ is a completely reducible $\pi^{-1}(H)-$module. Now,  we use the correspondence between $\pi^{-1}(H)-$stable subspaces of $(F^\delta)^n$ and $H-$stable ideals of $\mathrm{M_n}(F^\delta)$ as described in \cite[Proposition 4.1]{MMVRS-1} to prove (\ref{H-cr}). The assertions (\ref{H-irr})  and (\ref{H-ind}) are proved similarly.
  \end{proof}

 Let $A$ be a $\delta-F-$algebra and $K$ be a Picard Vessiot extension of $F$ that splits $A.$  Then, we know that $\mathrm G(K|F)$ acts as $\delta-F-$algebra automorphisms  on $(A\otimes K)^\delta\cong M_n(F^\delta),$ which in turn induces a  projective representation $\phi: \mathrm G(K|F)\to \mathrm{PGL_n}(F^\delta).$ We shall denote $\mathrm M_n(F^\delta)$ with this action of $\mathrm G(K|F)$ by $(M_n(F^\delta),G_\phi).$ Then, as in the context of characteristic zero differential fields, we have an exact analog of Tannakian formalism in the $\delta-$Galois theory of iterated differential modules (See \cite[$\S 13$]{MvdP03}, \cite{Nagy-Szamuely}). Then \cite[Proposition 3.1]{MMVRS-1}, in the context of iterative derivations, provide  a bijection between the collection of $\mathrm{G}(K|F)-$stable right ideals of $M_n(F^\delta)$ and right $\delta-$ideals\footnote{An ideal $I$ of $A$ is a \emph{$\delta-$ideal} of $A$ if for each $i\geq 0,$ $\delta^{(i)}_A$ maps $I$ to itself.} of $A.$

 We now make the following definitions. If  $A$ is a direct sum of minimal $\delta-$right ideals then $A$ is said to be  \emph{$\delta-$completely reducible}. If $A$ has no proper $\delta-$right ideal then we $A$ is said to be  \emph{$\delta-$irreducible}. If $A$ is not a direct sum of proper $\delta-$right ideals then we $A$ is said to be \emph{$\delta-$indecomposable}.
 
 \begin{theorem}
 	Let $(A,\delta)$ be a $\delta-F-$algebra with $F^\delta$ being an algebraically closed field.  Let $K$ be the Picard-Vessiot extension for  $\delta-F-$module $A$ and $\phi:\mathrm G(K/F)\rightarrow \mathrm{PGL_n}(F^\delta)$ be the corresponding projective representation. Then, $A$ is $\delta-$completely reducible (respectively, irreducible, indecomposable)  if and only if $\mathrm{G}(K|F)$ is $\mathrm{PGL_n}(F^\delta)-$completely reducible (respectively, irreducible, indecomposable)
 \end{theorem}
 \begin{proof}
 	Use Lemma \ref{PGL_n-cr-irr-ind} with $H=\phi(G(K|F))$ in conjunction with 
    \cite[Proposition 3.1]{MMVRS-1}.
 \end{proof}
 
\begin{corollary}
	If $D$ is $\delta-F-$division algebra, then the $\delta-$Galois group of $D$ is $\mathrm{PGL_n}-$irreducible and  in particular, it is reductive. 
	\end{corollary}

     \begin{proof} 
         From Lemma \ref{PGL_n-cr-irr-ind} (\ref{H-irr}), we know that the image of the projective representation of the $\delta-$Galois group is $\mathrm{PGL_n}-$irreducible. Since $\mathrm{PGL_n}$ is reductive, it follows that the $\delta-$Galois group is also reductive. 
     \end{proof}

{\bf Acknowledgement.}  The authors would like to thank Professors A. Merkurjev of University of California, Los Angeles and D. Hoffmann of Technische Universit\"{a}t for their helpful suggestions and remarks.

    \bibliographystyle{alpha}
    \bibliography{MMVRS}
  \end{document}